\documentclass{amsproc}
\usepackage{amsmath}
\usepackage{amssymb,amsfonts}
\usepackage{amsthm}
\usepackage{amsrefs}
\usepackage[english]{babel}
\usepackage{enumitem}
\usepackage{scalerel}
\usepackage[T1]{fontenc}
\usepackage{xcolor} 
\usepackage[pdfusetitle]{hyperref} 
\hypersetup{
	colorlinks,
	linkcolor={blue!80!black},
	citecolor={blue!50!black},
	urlcolor={blue!80!black},
	bookmarksnumbered,
	bookmarksdepth=2,
	pdfencoding=unicode|auto,
	pdftitle={\@title},
	pdfauthor={\@author}
}

\makeatletter
	\def\thesubsection{\thesection\ifnum\c@subsection=0\relax\else\,\fi\Alph{subsection}}
\makeatother

\newcommand{\Thorn}{\textup{\TH}}

\DeclareMathOperator*{\bigast}{{\smash{\scalebox{2.5}{\raisebox{-.9ex}{$*$}}}}}
\newcommand{\sglq}[1]{\text{`}#1\text{'}}
\newcommand{\dblq}[1]{\text{``}#1\text{''}}
\newcommand{\dblqm}[1]{\ulcorner#1\urcorner}

\newtheorem{theorem}{Theorem}[subsection]
\newtheorem{definition}[theorem]{Definition}
\newtheorem{claim}{Claim}[theorem]
\newtheorem{result}[theorem]{Result}
\newtheorem{lemma}[theorem]{Lemma}
\newtheorem{conjecture}[theorem]{Conjecture}
\newtheorem{corollary}[theorem]{Corollary}

\title{Forcing More \(\mathsf{DC}\) Over the Chang Model Using the Thorn Sequence}
\author{James Holland}
\author{Grigor Sargsyan}
\date{2023-03-07}

\begin{document}
	
	\maketitle
	
	\begin{abstract}
			In the context of \(\mathsf{ZF}+\mathsf{DC}\), we force \(\mathsf{DC}_\kappa\) for relations on \(\mathcal{P}(\kappa)\) for arbitrarily large \(\kappa<\aleph_\omega\) over the Chang model \(\mathrm{L}(\mathrm{Ord}^\omega)\) making some assumptions on the thorn sequence defined by \(\Thorn_0=\omega\), \(\Thorn_{\alpha+1}\) as the least ordinal not a surjective image of \(\Thorn_\alpha^\omega\) and \(\Thorn_\gamma=\sup_{\alpha<\gamma}\Thorn_\alpha\) for limit \(\gamma\). These assumptions are motivated from results about \(\Theta\) in the context of determinacy, and could be reasonable ways of thinking about the Chang model. Explicitly, we assume successor points \(\lambda\) on the thorn sequence are strongly regular---meaning regular and functions \(f:\kappa^{<\kappa}\rightarrow \lambda\) are bounded whenever \(\kappa<\lambda\) is on the thorn sequence---and justified---meaning \(\mathcal{P}(\kappa^\omega)\cap \mathrm{L}(\mathrm{Ord}^\omega)\subseteq \mathrm{L}_{\lambda}(\lambda^\omega,X)\) for some \(X\subseteq \lambda\) for any \(\kappa<\lambda\) on the thorn sequence.  This allow us to use Cohen forcing and establish more dependent choice while preserving the thorn sequence and calculating it: \(\Thorn_i=\aleph_{i+1}\) for \(0<i<\omega\) after \(i\) steps in the iteration.
	\end{abstract}
	
	\section{Introduction}
	
	The Chang model has had some recent research studying its connections with \(\mathsf{DC}\)~\cite{LarsonSargsyan} and \(\mathsf{AD}\)~\cite{TakehikoSargsyan}. But it is not nearly as well explored as models like \(\mathrm{L}(\mathbb{R})=\mathrm{L}(\omega^\omega)\) and its forcing extensions.  In the context of determinacy, there has been progress in how to think about \(\mathrm{L}(\omega^\omega)\) by way of assumptions like the regularity of \(\Theta\), and \(\mathsf{AD}\).  Such assumptions become standard ways of thinking and approaching the model.  Unfortunately, such standard assumptions don't seem to exist for the Chang model \(\mathrm{L}(\mathrm{Ord}^\omega)\) even though its theory cannot be changed by forcing under certain large cardinal assumptions~\cite[Corollary 3.1.7]{LarsonTower}.  Here we present some assumptions and arguments that use them.  In particular, we introduce a sequence called the \emph{thorn sequence} of \(\Thorn_\alpha\)---the letter \sglq{\TH} is called \emph{thorn}---for \(\alpha\in \mathrm{Ord}\), and force two statements simultaneously: for \(N<\omega\),
	\begin{enumerate}
		\item \(\Thorn_n=\aleph_n^+\) for \(n\le N\);
		\item \(\mathsf{DC}_{\kappa}\) for relations on \(\mathcal{P}(\kappa)\) for \(\kappa<\aleph_N\).
	\end{enumerate}
	The two assumptions used concern regularity and the calculation of powersets, similar to \(\Theta=\Thorn_1\) in the context of determinacy.
	
	The result runs contrary to anti-choice results related to the Chang model.  For example, if there are uncountably many measurable cardinals, then choice fails in \(\mathrm{L}(\mathrm{Ord}^\omega)\) by \cite{Kunenmeas} and this cannot be changed by forcing when there is a proper class of Woodin cardinals~\cite{LarsonTower}.  So while it is unlikely that full \(\mathsf{AC}\) can be forced over \(\mathrm{L}(\mathrm{Ord}^\omega)\) in the presence of many large cardinals, partial results are still possible.
	
	The main theorem is the following.
	
	\begin{theorem}\label{mainthmfinitethorn}
		Work in \(\mathrm{V}=\mathrm{L}(\mathrm{Ord}^\omega)\vDash\mathsf{ZF}+\mathsf{DC}\), and let \(n<\omega\). Suppose \(\Thorn_i\) is justified and strongly regular for each \(i\in \omega\). Let \(\mathbb{P}_n=\bigast_{i<n}\dot{\mathbb{Q}}_i\) be the iteration where \(\mathbb{Q}_i=\mathrm{Add}(\aleph_i^+,1)\), and let \(G_n\) be \(\mathbb{P}_n\)-generic over \(\mathrm{V}\). Then in \(\mathrm{V}[G_n]\),
		\begin{enumerate}[label=\roman*.]
			\item \(\mathsf{DC}_{\kappa}\) holds for relations on \(\mathcal{P}(\kappa)\) for all \(\kappa\le\aleph_n\).
			\item \(\Thorn_i^{\mathrm{V}}=\Thorn_i^{\mathrm{V}[G_n]}\) are still strongly regular for all \(i\in \omega\).
			\item \(\Thorn_0=\omega\), and \(\Thorn_{i}=\aleph_i^+\) for all \(0<i\le n\).
			\item \(\mathsf{GCH}\) holds below \(\aleph_n\): \(\kappa^+=2^\kappa\) for \(\kappa<\aleph_n\).
		\end{enumerate}
	\end{theorem}
		
	Let us first define the sequence of thorns.  Throughout this paper, we work merely in the context of \(\mathsf{ZF}+\mathsf{DC}\).
	
	\begin{definition}\label{thorndef}
		For sets \(A,B,X\),
		\begin{itemize}
			\item Write \(A\twoheadrightarrow B\) iff there is a surjection \(f:A\rightarrow B\).
			\item Define \(\aleph^*(X)\) to be the least ordinal such that \(X\not\twoheadrightarrow \aleph^*(X)\), i.e.\ 
			\[\aleph^*(X)=\sup\{\alpha+1\in \mathrm{Ord}:X\twoheadrightarrow \alpha\}\text{.}\]
			\item The \emph{thorn sequence} is the sequence of ordinals defined by
			\[\Thorn_0=\omega\text{,}\qquad \Thorn_{\alpha+1}=\aleph^*(\Thorn_\alpha^\omega)\text{,}\qquad \Thorn_\gamma=\bigcup\nolimits_{\alpha<\gamma}\Thorn_\alpha\]
			for \(\gamma\) a limit.
		\end{itemize}
	\end{definition}
	Hence \(\Thorn_1=\Theta\).  The primary assumptions we will be making are the following.
	\begin{definition}\label{thornjustifieddef}
		Let \(X\) be a set and \(\kappa\in\mathrm{Ord}\).
		\begin{itemize}
			\item We write \(\mathop{\mathrm{cof}}(\kappa)>X\) iff there is no function \(f:X\rightarrow \kappa\) such that \(\mathop{\mathrm{im}}(f)\) is unbounded in \(\kappa\).
			\item We say \(\kappa\) is \emph{strongly regular}\footnotemark{} iff \(\kappa\) is regular and \(\mathop{\mathrm{cof}}(\kappa)>\Thorn^{<\Thorn}\) whenever \(\Thorn<\kappa\) is on the thorn sequence.
			\item For \(\kappa\) on the thorn sequence, we say \(\kappa\) is \emph{justified} iff for any \(\Thorn<\kappa\) on the thorn sequence, \(\mathcal{P}(\Thorn^\omega)\subseteq \mathrm{L}_{\kappa}(\kappa^\omega,X)\) for some \(X\subseteq \kappa\).\footnotemark{}
		\end{itemize}
	\end{definition}
	\footnotetext[1]{This is a different notion of \dblq{strong regularity} from \cite{LarsonSargsyan}.}
	\footnotetext[2]{Here we organize the levels such that \(\mathrm{L}_{\kappa}(\kappa^\omega,X)\) has ordinal height \(\kappa\).}
	The motivation for these comes from determinacy where \(\Theta\) being regular implies it is strongly regular as above---just by virtue of \(\omega^{<\omega}=\omega\) and \(\omega<\mathop{\mathrm{cof}}(\Theta)\) in \(\mathsf{AD}\)---and obviously \(\Theta\) is justified in the context of \(\mathrm{L}(\mathbb{R})\).
	\begin{conjecture}\label{conjecturesuccessor}
		Assume a supercompact cardinal exists and \(\mathrm{V}\not=\mathrm{L}(X^\omega)\) for any set \(X\). Then for \(\alpha\ge 1\) and \(\mathop{\mathrm{cof}}(\alpha)\neq\omega\), \(\Thorn_{\alpha+1}=\Thorn_\alpha^+\) is regular.
	\end{conjecture}
	In \(\mathrm{L}(\mathbb{R})\) the conclusion is true, but it's less clear in a Chang model where \(\mathrm{L}(\mathrm{Ord}^\omega)\neq\mathrm{L}(\alpha^\omega)\) for all \(\alpha\in\mathrm{Ord}\).  
	
	There are a variety of questions one can ask about the thorn sequence and variations of the Chang model like the following: under reasonable assumptions,
	\begin{enumerate}
		\item Does \(\mathrm{L}(\Thorn_2^\omega)\vDash\mathrm{V}=\mathrm{L}(\Thorn_1^\omega)\) always?
		\item Does \(\mathrm{L}(\Thorn_2^\omega)\vDash\Thorn_2=\Thorn_1^+\) always?  Such a counter-example might give a counter-example to (1).
		\item Does \(\mathsf{DC}\) hold in \(\mathrm{L}(\bigcup_{n<\omega}\Thorn_n^\omega)\)?
	\end{enumerate}
	We suspect that large cardinals imply that \(\kappa^+\) is regular for \(\kappa\ge\Thorn_1\). Of course, assuming \(\Thorn_1^+\) is strongly regular, we get a positive answer to (2).  Also, for now we will focus on the question of getting higher amounts of \(\mathsf{DC}\) via forcing, which will preserve members of the thorn sequence, but change the calculations of the \(\aleph_n\)s for \(n<\omega\).
	
	Firstly, in our framework of assuming \(\mathsf{DC}\), we get \(\mathsf{DC}\) in the Chang model and its varients: \(\mathrm{L}(\kappa^\omega)\) for any \(\kappa\in \mathrm{Ord}\) satisfies \(\mathsf{DC}\) in a nearly identical way as in Proposition 11.13 of \cite{Kanamori}.
	\begin{result}\label{DCinLω}
		Assume \(\mathsf{ZF}+\mathsf{DC}\).  Let \(\kappa\in \mathrm{Ord}\).  Then \(\mathrm{L}(\kappa^\omega)\vDash\mathsf{DC}\).
	\end{result}
	\begin{proof}
		Let \(R\) be a relation such that any finite length \(R\)-chain can be extended.   We aim to show that there is an infinite branch for \(R\) in \(\mathrm{L}(\kappa^\omega)\).  We know there is one in \(\mathrm{V}\), and so by the absoluteness of well-foundedness, \(R^{-1}\) is illfounded.  But unfortunately, illfoundedness doesn't give us an infinite branch generally unless we already have \(\mathsf{DC}\).  So instead we need to work with ordinals to be able to choose one.
		
		In \(\mathrm{L}(\kappa^\omega)\), there is a canonical surjection \(F:\mathrm{Ord}\times\kappa^\omega\rightarrow \mathrm{L}(\kappa^\omega)\) just given by the construction of the levels of \(\mathrm{L}(\kappa^\omega)\).  Let \(X=\langle x_n:n<\omega\rangle\) be an infinite \(R\)-branch in \(\mathrm{V}\), and let \(\chi\in \mathrm{Ord}\) be such that \(\mathop{\mathrm{im}}(X)\subseteq F"(\chi\times\kappa^\omega)\).  Now again by \(\mathsf{DC}\), for \(n<\omega\), choose \(y_n\in \kappa^\omega\) such that \(\exists\xi_n<\chi\ F(\xi_n,y_n)=x_n\).  The point now is to canonically choose these \(\xi_n\)s without knowing what \(x_n\) is.  We do, however, have that \(\langle y_n:n<\omega\rangle\in \kappa^{\omega\times\omega}\approx\kappa^\omega\) is in \(\mathrm{L}(\kappa^\omega)\).
		
		Now consider a relation \(R_F\in \mathrm{L}(\kappa^\omega)\) on \(\chi\times\omega\) defined by
		\[\langle \xi_0,n_0\rangle  \mathrel{R_F} \langle \xi_1,n_1\rangle \quad\text{iff}\quad n_0=n_1+1\wedge \langle F(\xi_0,y_{n_1}),F(\xi_1,y_{n_1+1})\rangle \in R\]
		It follows that this relation is ill-founded in \(\mathrm{V}\) (by \(\mathsf{DC}\) to choose the relevant ordinals) and hence in \(\mathrm{L}(\kappa^\omega)\) by the absoluteness of ill-foundedness over \(\mathsf{ZF}\).  So let \(S\subseteq \chi \times \omega\) in \(\mathrm{L}(\kappa^\omega)\) be non-empty with no \(R_F\)-minimal element. Let \(\langle \xi_0,n_0\rangle \in S\) be arbitrary, and then we can define \(\langle \xi_m,n_m:m<\omega \rangle\) by taking \(\xi_{m+1}\) to be the least ordinal such that \(\langle \xi_{m+1},n_0+m+1\rangle \mathrel{R_F}\langle \xi_{m},n_0+m\rangle\) and \(\langle \xi_{m+1},n_0+m+1\rangle \in S\).  The resulting sequence gives \(\langle F(\xi_m,y_{n_m}):n<\omega \rangle \in \mathrm{L}(\kappa^\omega)\) as an infinite \(R\)-chain, as desired.\qedhere
	\end{proof}

	This of course generalizes to the Chang model \(\mathrm{L}(\mathrm{Ord}^\omega)\).  The closure under \(\omega\)-length sequences of \(\kappa\) is crucial to this proof, and is why one might suspect \(\mathsf{DC}\) could fail in \(\mathrm{L}(\bigcup_{n<\omega}\Thorn^\omega_n)\), where we only have access to \emph{bounded} sequences in \(\Thorn_\omega^\omega\).
	
	The goal from here on is to establish more and more \(\mathsf{DC}\) in forcing extensions of \(\mathrm{L}(\mathrm{Ord}^\omega)\).  The first step of this is establishing \(\mathsf{DC}_{\omega_1}\) after forcing with \(\mathrm{Add}(\omega_1,1)\) under the assumptions of \(\Thorn_1\) being strongly regular and justified.
	
	\section{The First Steps}

	The main strategy here is as follows, working over \(\mathrm{V}=\mathrm{L}(\mathrm{Ord}^\omega)\):
	\begin{itemize}
		\item Force with \(\mathrm{Add}(\omega_1,1)\).  Assuming \(\Thorn_1\) is strongly regular, we get \(\Thorn_1=\aleph_2\) and \(\mathsf{DC}_{\omega_1}\) for relations on \(\omega_1^\omega\approx \omega_1\).
		\item Assuming \(\Thorn_1\) is also justified, we get \(\mathsf{DC}_{\omega_1}\) for relations on \(\mathcal{P}(\omega_1)\).
		\item Inductively, we then use \(\mathsf{DC}_\kappa\) for relations on \(\mathcal{P}(\kappa)\) to force with \(\mathrm{Add}(\kappa^+,1)\) and get \(\aleph^*(\kappa^\omega)=\aleph^*(\kappa)=\kappa^+\) assuming strong regularity.
		\item \(\kappa^{++}\) being justified tells us \(\mathsf{DC}_{\kappa^+}\) holds for relations on \(\mathcal{P}(\kappa^+)\) so we can continue the induction.
	\end{itemize}
	A fair amount of this is combinatorial, which is why strong regularity is phrased the way it is.  Being justified also helps in finding certain surjections: if \(\mathcal{P}(\kappa)\subseteq \mathrm{L}_\lambda (\lambda^\omega,X)\) for \(X\subseteq \lambda \ge\omega\), then canonically \(\lambda^\omega\approx \lambda \times \lambda^\omega\times X\twoheadrightarrow \mathcal{P}(\kappa)\), meaning there is a canonical surjection witnessing this.  Additionally, we require larger and larger amounts of \(\mathsf{DC}\) to continue this sort of argument because we require a certain amount of distributivity from the Cohen forcings.  Although \(\mathrm{Add}(\kappa^+,1)\) is \(\le\kappa\)-closed, \emph{a priori} without choice, we don't know this means \(\mathrm{Add}(\kappa^+,1)\) is \(\le\kappa\)-distributive.
	
	Note that we don't need full \(\mathsf{DC}_\kappa\), just on \(\mathrm{Add}(\kappa^+,1)\).  The following definition makes this precise.
	\begin{definition}\label{DCvariantdef}
		Let \(X\) be a set, \(\kappa \in \mathrm{Ord}\).  \(\mathsf{DC}_\kappa\) for relations on \(X\) is the statement that for any relation \(R\subseteq X^2\) such that \(<\kappa\)-length \(R\)-chains have upper bounds in \(R\), there is a \(\kappa\)-length \(R\)-chain.
	\end{definition}

	It will be useful in what follows to note that we can canonically identify \(\mathrm{Add}(\kappa^+,1)\) with \(\mathcal{P}(\kappa)\).\footnote{To see this, we can code pairs \(\langle \alpha,\beta,\gamma \rangle \in \kappa^2\times 2\) as a single ordinal \(\mathop{\mathrm{code}}(\alpha,\beta,\gamma)\in \kappa\) via a bijection \(\mathop{\mathrm{code}}:\kappa^2\times 2\rightarrow \kappa\).  We may identify well-orders of \(\kappa\), i.e.\ ordinals \(<\kappa^+\), by certain subsets of \(\kappa^2\).  In particular, for any \(X\subseteq \kappa\), consider \(\mathop{\mathrm{code}}^{-1}"X=\{\langle \langle \alpha,\beta\rangle ,\gamma \rangle \in \kappa^2\times 2:\mathop{\mathrm{code}}(\alpha,\beta,\gamma)\in X\}\).  If \(D=\mathop{\mathrm{dom}}(\mathop{\mathrm{code}}^{-1}"X)\) codes a well-order of length \(\alpha\) via ordertype function \(\pi :D\rightarrow \alpha\)---which is unique---and if \(\mathop{\mathrm{code}}^{-1}"X\) is a function, then define \(f(X)\) as the function \(\{\langle \pi (x),\gamma \rangle :\langle x,\gamma \rangle \in \mathop{\mathrm{code}}^{-1}"X\}\).  We can then consider \(X\leqslant Y\) iff \(f(X)\supseteq f(Y)\).  This \(f:\mathcal{P}(\kappa)\rightarrow \mathrm{Add}(\kappa^+,1)\) is a surjection, and the preorder (rather than poset) \(\langle \mathcal{P}(\kappa),\leqslant\rangle\) is equivalent to \(\mathrm{Add}(\kappa^+,1)\).}
	
	\begin{lemma}\label{adddistrib}
		Suppose \(\mathsf{DC}_{\kappa}\) holds for relations on \(\mathcal{P}(\kappa)\) in \(\mathrm{V}\).  Then \(\mathbb{P}=\mathrm{Add}(\kappa^+,1)\) does not add any functions \(f:\kappa \rightarrow \mathrm{V}\).
	\end{lemma}
	\begin{proof}
		Identify \(\mathbb{P}\) with \(2^{<\kappa^+}\) and thus with \(\mathcal{P}(\kappa)\) so that we may use \(\mathsf{DC}_\kappa\) with it.  Let \(G\) be \(\mathrm{Add}(\kappa^+,1)\)-generic over \(\mathrm{V}\). Suppose \(f:\kappa \rightarrow \mathrm{V}\) is in \(\mathrm{V}[G]\).  Let \(\dot f\) be a name for \(f\).  Let \(p\) force that \(\dot f\) is a function with domain \(\kappa\).  Now consider the tree of conditions of \(\mathbb{P}\) below \(p\) that decide more and more of \(\dot f\).  More precisely, for each \(p^*\leqslant p\), let \(f_{p^*}\in \mathrm{V}\) denote the largest initial segment of \(\dot f\) decided by \(p^*\): \(p^*\Vdash\dblqm{\check f_{p^*}\subseteq \dot f}\) and \(\mathop{\mathrm{dom}}(f_{p^*})\in \mathrm{Ord}\).  We then consider \(T\) as equivalence classes of \(p^*\leqslant p\) modulo equality of \(f_{p^*}\).  We may consider \(T\) as a relation on \(\mathcal{P}(\kappa)\).  Then we order \(T\) by \(q\prec r\) iff \(f_q \subsetneq f_r\).  We will find a condition \(p^*\) such that \(p^*\Vdash\dblqm{\check f_{p^*}=\dot f}\).
			
		Suppose \(\vec{x}\) is any chain of length \(<\kappa^+\).  If \(\mathop{\mathrm{lh}}(\vec{x})\) is a limit, then \(x=\bigcup_{\alpha<\mathop{\mathrm{lh}}(\vec{x})}\vec{x}(\alpha)\) is a condition of \(\mathrm{Add}(\kappa^+,1)\) such that \(f_x\supseteq \bigcup_{\alpha<\mathop{\mathrm{lh}}(\vec{x})}f_{\vec{x}(\alpha)}\) is strictly larger than any of the previous initial segments. Now suppose \(\mathop{\mathrm{lh}}(\vec{x})=\alpha+1\) is a successor ordinal.  Since \(f:\kappa \rightarrow \mathrm{V}\), we have \(f(\alpha+1)=y\in \mathrm{V}\) for some \(y\) as forced by some \(p^*\in \mathbb{P}\): \(p^*\Vdash\dblqm{\dot f(\check \alpha+1)=\check y}\).  Hence \(f_{p^*}\supseteq f\cup \{\langle \alpha+1,y\rangle \}\) allows us to extend \(\vec{x}\).
			
		The result is that for any \(p^*\leqslant p\) in \(\mathbb{P}\), by \(\mathsf{DC}_{\kappa}\), we can always find a branch of \(T\) below \(p^*\) that decides all of \(f\).  By density, a single \(p^*\in G\) will decide all of \(f\), and using this \(p^*\), we can define \(f\) in \(\mathrm{V}\).\qedhere
	\end{proof}

	For the following, we don't need that \(\Thorn_1\) is justified; we only use strong regularity.
	
	\begin{theorem}\label{changthorn1}
		Work in \(\mathrm{V}=\mathrm{L}(\mathrm{Ord}^\omega)\vDash\mathsf{ZF}+\mathsf{DC}\).  Suppose \(\Thorn_1\) is strongly regular.  Let \(G\) be \(\mathbb{P}_0=\mathrm{Add}(\Thorn_0^+,1)\)-generic over \(\mathrm{V}\).  Then in \(\mathrm{V}[G]\),
		\begin{enumerate}
			\item \(\Thorn_1^{\mathrm{V}}=\Thorn_1^{\mathrm{V}[G]}\);
			\item \(|\omega^\omega|=\aleph_1\) and \(\Thorn_1=\aleph_1^+=\aleph_2\);
			\item \(\Thorn_1\) is still regular.
			\item In fact, any \(\kappa >\aleph_1\) that is strongly regular in \(\mathrm{V}\) is still strongly regular in \(\mathrm{V}[G]\), e.g.\ \(\Thorn_1\).
			\item Moreover, \(\Thorn_\alpha^{\mathrm{V}}=\Thorn_{\alpha}^{\mathrm{V}[G]}\) for any \(\alpha\).
			\item \(\mathsf{DC}_{\omega_1}\) holds for relations on \(\Thorn_1^\omega\).
		\end{enumerate}
	\end{theorem}
	\begin{proof}
		Note that there is a canonical class surjection \(F:\mathrm{Ord}^\omega\rightarrow \mathrm{V}\) and \(F_G:\mathrm{Ord}^\omega\rightarrow \mathrm{V}[G]\) by constructibility.  Since \(\mathsf{DC}\) holds by \autoref{DCinLω}, \autoref{adddistrib} implies we don't add new countable sequences.  So \(\omega_1^{\mathrm{V}}=\omega_1^{\mathrm{V}[G]}\), and \((\omega^\omega)^{\mathrm{V}}=(\omega^\omega)^{\mathrm{V}[G]}\), for example.
		\begin{claim}\label{changthorn1cl1}
			In \(\mathrm{V}\), \(\omega_1^\omega\approx \mathrm{Add}(\omega_1,1)\not\twoheadrightarrow\Thorn_1^{\mathrm{V}}\).  In particular, \(\Thorn_1>\omega_1\) in \(\mathrm{V}\).
		\end{claim}
		\begin{proof}
			We have \(\omega^\omega\twoheadrightarrow\omega_1^\omega\) by the following: take \(r\in \omega^\omega\approx \omega^{\omega\times \omega}\approx(\omega^\omega)^\omega\) and identify it with infinitely many reals \(r_n\in \omega^\omega\), \(n<\omega\).  If \(r_n\) codes a well-order of \(\omega\), send \(r_n\) to its order type \(\mathrm{ot}(r_n)\in \omega_1\).  Otherwise, send \(r_n\) to \(\mathrm{ot}(r_n)=r_n(0)\).  It follows that the map \(f\) sending \(r\in \omega^\omega\) to \(\langle \mathrm{ot}(r_n):n<\omega\rangle \in \omega_1^\omega\) is surjective.  In particular, if there is a surjection \(g:\omega_1^\omega\rightarrow \Thorn_1\), then \(g\circ f:\omega^\omega\rightarrow \Thorn_1\) would contradict that \(\omega^\omega\) doesn't surject onto \(\Thorn_1\).\qedhere
		\end{proof}
		Now we may show the results of the theorem.
		\begin{enumerate}\parskip 10pt
			\item Suppose \(f:\omega^\omega\rightarrow \Thorn_1^{\mathrm{V}}\) is in \(\mathrm{V}[G]\).  We want to show \(f\) isn't surjective. There is a name \(\dot f\) for \(f\) forced to have domain \((\omega^\omega)^{\mathrm{V}}\) (below some arbitrary condition).  Consider \(f':\omega^\omega\times \omega_1^\omega\rightarrow \Thorn_1^{\mathrm{V}}\) where \(f'(x,p)\) is the least \(\gamma <\Thorn_1\) such that \(p\Vdash\dblqm{\dot f(\check x)=\check\gamma}\).  Through coding pairs, this gives a map \(f':\omega_1^\omega\rightarrow \Thorn_1^{\mathrm{V}}\), which cannot be surjective by \autoref{changthorn1cl1}, implying \(f\) isn't surjective in \(\mathrm{V}[G]\).
			
			\item It's not hard to see \(\mathrm{Add}(\omega_1,1)\) adds a well-order of \(\omega^\omega=\Thorn_0^\omega\) of length \(\omega_1\) because consider \(g\) as what \(G\) says about \(\omega^\omega\):
			\[g=\{\langle \alpha,r\rangle :\exists p\in G\ (\langle \alpha,r\rangle \in p\wedge r\in \omega^\omega)\}=(\omega_1\times \omega^\omega)\cap \bigcup G\text{.}\]
			This \(g\) is a surjection from an uncountable subset of \(\omega_1\) to \(\omega^\omega\), as witnessed by
			\begin{align*}
			D_{\ge\alpha}&=\{p\in \mathbb{P}:\exists \beta\ge\alpha\ (p(\beta)\in \omega^\omega)\}\\
			E_r&=\{p\in \mathbb{P}:r\in \mathop{\mathrm{im}}(p)\}\text{,}
			\end{align*}
			which are dense for each \(\alpha<\omega_1\) and \(r\in \omega^\omega\).  So now we recursively define two sequences \(\langle \beta_\alpha<\omega_1:\alpha<\omega_1\rangle\) and \(\langle r_\alpha\in \omega^\omega:\alpha<\omega_1\rangle\):
			\begin{itemize}
				\item \(\beta_0=\min(\mathop{\mathrm{dom}}(g))\) and \(r_0=g(\beta_0)\).
				\item For \(\alpha>0\), \(r_\alpha=g(\beta_\alpha)\); and
				\item \(\beta_\alpha\) is the least element \(\beta\in \mathop{\mathrm{dom}}(g)\) such that \(\beta>\beta_\xi\) for each \(\xi <\alpha\) and \(g(\beta)\not\in \{r_\xi :\xi <\alpha\}\).
			\end{itemize}
			It follows that \(\{r_\alpha:\alpha<\omega_1\}=\mathop{\mathrm{im}}(g)=\omega^\omega\) and so we have a bijection from \(\omega_1\) to \(\omega^\omega\). It follows that \(\Thorn_1=\aleph_1^+=\aleph_2\).
			
			\item Suppose \(f:\omega_1\rightarrow \Thorn_1\) has an unbounded image in \(\mathrm{V}[G]\).  We can then consider a name \(\dot f\) for \(f\) and a function \(f':\omega_1\times \omega_1^{\omega}\rightarrow \Thorn_1\) defined by \(f'(\alpha,p)=\beta\) iff \(p\Vdash\dblqm{\dot f(\check\alpha)=\check\beta}\) (and otherwise \(f'(\alpha,p)=0\)).  This would imply \(f':\omega_1^{<\omega_1}\rightarrow \Thorn_1\) is unbounded.  Note that we can find a surjection \(g:\omega^\omega\rightarrow \omega_1^\omega\) so that composing these gives \(f'\circ g:\omega^\omega\rightarrow \Thorn_1\) an unbounded map from \(\omega^\omega=\Thorn_0^\omega\) to \(\Thorn_1\) which contradicts strong regularity in \(\mathrm{V}\).
			
			\item Suppose \(\kappa\) is strongly regular in \(\mathrm{V}\).  Suppose \(f:\lambda^{<\lambda}\rightarrow \kappa\) in \(\mathrm{V}[G]\) is unbounded in \(\kappa\) for some \(\lambda <\kappa\).  We as before can then get an unbounded \(f':\lambda^{<\lambda}\times \omega_1^{\omega}\rightarrow \kappa\).  Since \(\kappa >\omega_1\), without loss of generality, \(\lambda \ge\omega_1\) so by coding \(\omega_1^\omega\) into \(\lambda^{<\lambda}\), \(f'\) is equivalent to an unbounded function from \(\lambda^{<\lambda}\) into \(\kappa\), and hence in \(\mathrm{V}\), we violate strong regularity.
			
			\item This same proof shows that the thorn sequence is preserved: we just showed \(\Thorn_1\) is the same in both \(\mathrm{V}\) and \(\mathrm{V}[G]\).  Inductively, if \(\Thorn_\alpha\) is the same in both, for \(\Thorn_{\alpha+1}\) with \(\alpha\ge 1\), any function \(f:\Thorn_\alpha^\omega\rightarrow \mathrm{Ord}\) in \(\mathrm{V}[G]\) can be expanded to a function \(f:\Thorn_\alpha^\omega\times \omega_1^\omega\rightarrow \mathrm{Ord}\) by way of \(f'(x,p)=\beta\) iff \(p\Vdash\dblqm{\dot f(\check x)=\check\beta}\) where \(\dot f\) is a name for \(f\).  Since \(\alpha\ge 1\), \(\Thorn_\alpha\ge\omega_1\) and we can code to get \(f':\Thorn_\alpha^\omega\rightarrow \mathrm{Ord}\).  If \(f\) surjected onto \(\Thorn_{\alpha+1}^{\mathrm{V}}\), then \(f'\) would surject onto \(\Thorn_{\alpha+1}^{\mathrm{V}}\), contradicting that \(f'\in \mathrm{V}\).  Limit stages being absolute are obvious.  Hence by induction on \(\alpha\), \(\Thorn_\alpha^{\mathrm{V}}=\Thorn_{\alpha}^{\mathrm{V}[G]}\) for all \(\alpha\in \mathrm{Ord}\).
			
			\item Let \(R\) be a relation on \(\Thorn_1^\omega\) of height \(\le\omega_1\) where \(\Thorn_1>\omega_1\) is strongly regular.  Assume \(<\omega_1\)-length chains of \(R\) have upper bounds in \(R\).  We want to show there is an \(\omega_1\)-length branch of \(R\).
			
			Let \(\gamma_\beta\) be the least length such that all countable \(R\)-chains in \(\beta^\omega\) have an upper bound in \(\gamma_\beta^\omega\). Note that if \(\beta<\Thorn_1\) then \(\gamma_\beta<\Thorn_1\).  To see this, suppose \(\beta<\Thorn_1\) but \(\gamma_\beta\ge\Thorn_1\) (so that \(\gamma_\beta=\Thorn_1\)).  For \(x\in \beta^\omega\), let \(g(x)\) be the least \(\gamma <\Thorn_1\) such that \(x\) has a bound in \(\gamma^\omega\).  It follows that \(g\) must be unbounded in \(\Thorn_1\), contradicting its strong regularity.
			
			So let \(\hat\beta<\Thorn_1\) be sufficiently large of cofinality \(\omega_1\) such that \(\gamma_\xi \le\hat\beta\) for each \(\xi <\hat\beta\).  Such a \(\hat\beta\) exists just by taking the supremum of the sequence defined by recursion \(\eta_0=\gamma_0\), \(\eta_{\alpha+1}=\gamma_{\eta_\alpha}<\Thorn_1\), and taking supremums at limits \(\alpha<\omega_1\).  \(\hat\beta=\sup_{\alpha<\omega_1}\eta_\alpha<\Thorn_1\) by regularity, and has the desired closure property.
			
			We know \(\omega^\omega\twoheadrightarrow\hat\beta^\omega\), and \(|\omega^\omega|=\aleph_1\) in \(\mathrm{V}[G]\).  So let \(f:\omega_1\rightarrow \hat\beta^\omega\) be a surjection.  Consider the relation \(R'\subseteq \omega_1^2\) defined by
			\[x_0\mathrel{R'}x_1\quad\text{iff}\quad f(x_0) \mathrel{R} f(x_1)\text{.}\]
			It follows that \(R'\) is still \(<\omega_1\)-closed of height \(\le\omega_1\).  But the domain of \(R'\), as a subset of \(\omega_1\), can be well-ordered.  Hence we can find a \(\omega_1\)-length branch.  To see this, for \(\vec{x}\) a chain already constructed, pick the least \(\alpha\in \omega_1\) with \(\vec{x}^\frown\langle \alpha\rangle\) still an \(R'\)-chain.  By \(<\omega_1\)-closure, we can continue to find such an \(\alpha\) at any \(\vec{x}\) with countable length.  Taking unions at limit stages, the resulting construction continues to an \(\omega_1\)-length chain \(\langle x_\alpha:\alpha<\omega_1\rangle\) so that \(\langle f(x_\alpha):\alpha<\omega\rangle\) is an \(R\)-chain.  As a result, \(\mathsf{DC}_{\omega_1}\) holds on such relations.\qedhere
		\end{enumerate}
	\end{proof}

	Strong regularity comes in to get \(\mathsf{DC}_{\omega_1}\) on more relations, and in particular on \(\mathcal{P}(\Thorn_0^+)=\mathcal{P}(\omega_1)\) in this context.  This allows us to have \autoref{adddistrib} with \(\mathrm{Add}(\omega_2,1)\) which will then give us more \(\mathsf{DC}\).
	
	\begin{result}\label{changthorn1justified}
		Work in \(\mathrm{V}=\mathrm{L}(\mathrm{Ord}^\omega)\vDash\mathsf{ZF}+\mathsf{DC}\).  Suppose \(\Thorn_1\) is strongly regular and justified.  Let \(G\) be \(\mathbb{P}=\mathrm{Add}(\Thorn_0^+,1)\)-generic over \(\mathrm{V}\).  Then in \(\mathrm{V}[G]\), \(\mathsf{DC}_{\omega_1}\) holds for relations on \(\mathcal{P}(\omega_1)\).
	\end{result}
	\begin{proof}
		Because \(\Thorn_1\) is justified, there's a canonical surjection witnessing
		\[\Thorn_1^\omega\approx \Thorn_1^\omega\times X\times \Thorn_1\twoheadrightarrow \mathcal{P}(\omega^\omega)\approx \mathcal{P}(\omega^\omega\times \omega_1^\omega)\approx\mathcal{P}(\omega_1^\omega)\]
		in \(\mathrm{V}\) where \(X\subseteq \Thorn_1\).  In particular, in \(\mathrm{V}[G]\), there is a surjection \(f:\Thorn_1^\omega\times X\rightarrow \mathcal{P}(\omega_1)\). By identifying \(X\) with its order-type \(\mathrm{ot}(X)\le\Thorn_1\), we may regard \(\mathop{\mathrm{dom}}(f)\) as instead just \(\Thorn_1^\omega\).  In \(\mathrm{V}[G]\), if \(R\) is a relation on \(\mathcal{P}(\omega_1)\), we can then consider the pull-back relation defined by
		\[\vec{x}\mathrel{R'}\vec{y}\quad\text{iff}\quad f(\vec{x})\mathrel{R}f(\vec{y})\text{.}\]
		Let \(\langle x_\alpha:\alpha<\tau\rangle\) be an \(R\)-chain where \(\tau<\omega_1\).  By \(\mathsf{DC}\), we can pull this back to an \(R'\)-chain \(\langle y_\alpha:\alpha<\tau\rangle\) where \(f(y_\alpha)=x_\alpha\).  \(R'\) is a relation on \(\Thorn_1^\omega\) so by \autoref{changthorn1} (6), \(\mathsf{DC}_{\omega_1}\) holds for \(R'\): we get an \(\omega_1\)-length \(R'\)-chain \(\langle y_\alpha:\alpha<\omega_1\rangle\).  This gives an \(\omega_1\)-length \(R\)-chain \(\langle x_\alpha:\alpha<\omega_1\rangle\).\qedhere
	\end{proof}

\section{Higher Amounts of \texorpdfstring{\(\mathsf{DC}\)}{DC}}
	
	To go beyond this, we will force with \(\mathrm{Add}(\omega_2,1)\) to get \(\mathsf{DC}_{\omega_2}\) on \(\Thorn_2^\omega\) in a way that doesn't change \(\Thorn_1=\omega_2\) and still preserves any strongly regular cardinals.  Assuming \(\Thorn_2\) is justified, we can expand this to \(\mathsf{DC}_{\omega_2}\) on \(\mathcal{P}(\omega_2)\).  Then we continue on in this way, forcing with \(\mathrm{Add}(\omega_{n+2},1)=\mathrm{Add}(\Thorn_n^+,1)\) to calculate \(\Thorn_{n+1}\) and get \(\mathsf{DC}_{\Thorn_{n+1}}\) on \(\Thorn_{n+1}^\omega\).  Assuming these thorns are justified, this expands to \(\mathsf{DC}_{\Thorn_{n+1}}\) on \(\mathcal{P}(\Thorn_{n+1})\) to give sufficient distributivity for the next forcing.  Before we get into this, however, we should note that small forcings do not change the thorn sequence assuming strong regularity for it.
	
	\begin{lemma}\label{changthornlatersteps}
		\begin{enumerate}[label=\alph*.]
			\item Let \(\kappa >\aleph_0\) be an arbitrary cardinal.
			\item Assume \(\mathsf{DC}_\kappa\) holds for relations on \(\mathcal{P}(\kappa)\).
			\item Let \(\mathbb{P}=\mathrm{Add}(\kappa^+,1)\).
			\item Suppose \(G\) is \(\mathbb{P}\)-generic over \(\mathrm{V}\vDash\mathsf{ZF}\).
			\item In \(\mathrm{V}\), let \(\alpha\in \mathrm{Ord}\), and suppose that for any \(\beta<\alpha\) such that \(\kappa <\Thorn_\beta\), if \(\beta\) is a successor, \(\Thorn_\beta\) is strongly regular.
			\item In \(\mathrm{V}\), suppose for any \(\beta\), \(\Thorn_\beta^\omega\) can be well-ordered to have size \(\le\kappa\) whenever \(\Thorn_\beta<\kappa\).
		\end{enumerate}
	\noindent Then, 
		\begin{enumerate}
			\item 
			\(\mathbb{P}=\mathrm{Add}(\kappa^+,1)\) adds a well-order of \(\mathcal{P}(\kappa)\) of length \(\kappa^+\) so that \((\kappa^+)^{<\kappa^+}=\kappa^+\) in \(\mathrm{V}[G]\).
			\item 
			\(\Thorn_\beta^{\mathrm{V}}=\Thorn_\beta^{\mathrm{V}[G]}\) for all \(\beta<\alpha\) whenever \(\Thorn_\beta<\kappa\).
			\item If \(\kappa^{<\kappa}\) can be well-ordered in \(\kappa\) in \(\mathrm{V}\) then \(\Thorn_{\beta}^{\mathrm{V}}=\Thorn_{\beta}^{\mathrm{V}[G]}\) for all \(\beta<\alpha\).
			\item 
			\(\mathbb{P}\) preserves strong regularity above \(\kappa^+\): if \(\lambda^{<\lambda}<\mathop{\mathrm{cof}}(\mu)\) for some \(\lambda <\mu\) in \(\mathrm{V}\), then this also holds in \(\mathrm{V}[G]\) for any \(\mu >\kappa^+\).
			\item 
			If \(\kappa\) is regular and \(\kappa^{<\kappa}=\kappa\) in \(\mathrm{V}\), then \(\mathbb{P}\) preserves strong regularity.
		\end{enumerate}
	\end{lemma}
	\begin{proof}
		\begin{enumerate} \parskip 10pt
			\item Proceed similarly to \autoref{changthorn1} (2).  Regard \(\mathbb{P}\) as \((2^{\kappa})^{<\kappa^+}\) so conditions are functions from ordinals \(\alpha<\kappa^+\) to \(2^\kappa\).  Consider \(g\) as what \(G\) says about \(2^\kappa\):
			\[g=\{\langle \alpha,r\rangle \in \kappa^+\times 2^\kappa  : \exists p\in G\ (\langle \alpha,r\rangle \in p)\}\]
			This \(g\) is a surjection from an unbounded subset of \(\kappa^+\) to \(2^\kappa\), as witnessed by
			\begin{align*}
				D_{\ge\alpha}&=\{\langle \alpha,r\rangle :\exists \beta\ge\alpha\ (p(\beta)\in 2^\kappa)\}\\
				E_r&=\{p\in \mathbb{P}:r\in \mathop{\mathrm{im}}(p)\}
			\end{align*}
			which are dense for each \(\alpha<\kappa^+\) and \(r\in 2^\kappa\).  So now we recursively define two sequences \(\langle \beta_\alpha<\kappa^+:\alpha<\gamma \rangle\), for some \(\gamma \le\kappa^+\), and \(\langle r_\alpha\in 2^\kappa :\alpha<\gamma \rangle\):
			\begin{itemize}
				\item \(\beta_0=\min(\mathop{\mathrm{dom}}(g))\) and \(r_0=g(\beta_0)\).
				\item For \(\alpha>0\), \(r_\alpha=g(\beta_\alpha)\); and
				\item \(\beta_\alpha\) is the least element \(\beta\in \mathop{\mathrm{dom}}(g)\) such that \(\beta>\beta_\xi\) and for each \(\xi <\alpha\) and \(g(\beta)\not\in \{r_\xi :\xi <\alpha\}\).
			\end{itemize}
			It follows that \(\{r_\alpha:\alpha<\gamma \}=\mathop{\mathrm{im}}(g)=2^\kappa\) and so we have a bijection from \(\gamma \le\kappa^+\) to \(2^\kappa\).  Since clearly \(|2^\kappa |=|\gamma |\ge \kappa^+\), we must have \(\gamma =\kappa^+\), and so \(2^\kappa =\kappa^+\) in \(\mathrm{V}[G]\).  It follows that
			\[(\kappa^+)^{<\kappa^+}=(\kappa^+)^\kappa =(2^\kappa)^\kappa =2^{\kappa \cdot \kappa}=\kappa^+\text{.}\]
			\item Proceed by induction on \(\beta<\alpha\).  For \(\beta=0\), this is clear.  For successor \(\beta+1<\alpha\), suppose \(\Thorn_\beta^{\mathrm{V}}<\kappa\).  Since we can well-order \(\Thorn_\beta^\omega\) of size \(\le\kappa\), distributivity implies we don't add any functions \(f:\Thorn_\beta^\omega\rightarrow \mathrm{Ord}\) going from \(\mathrm{V}\) to \(\mathrm{V}[G]\) and hence the calculation of \(\Thorn_{\beta+1}\) is the same in both. For limit \(\beta\), if \(\Thorn_\xi^{\mathrm{V}}=\Thorn_\xi^{\mathrm{V}[G]}\) for all \(\xi <\beta\), then \(\Thorn_\beta\), as the supremum of these, is the same in \(\mathrm{V}\) and \(\mathrm{V}[G]\).
			
			\item We already have the base case of \(\beta=0\), the limit case, and the case that \(\Thorn_{\beta}<\kappa\).  So inductively, suppose \(\Thorn_\beta^{\mathrm{V}}=\kappa\) and \(\kappa^{<\kappa}=\kappa\) in \(\mathrm{V}\).  Thus \(\aleph^*(\kappa^\omega)=\aleph^*(\kappa^{<\kappa})=\aleph^*(\kappa)=\kappa^+\). By distributivity, \(\mathrm{V}[G]\) adds no functions from \(\kappa\) to \(\mathrm{Ord}\) and so \(\mathrm{V}\), \(\mathrm{V}[G]\) agree on \(\aleph^*(\kappa^\omega)\) and \(\kappa^+\): \(\Thorn_{\beta+1}=\kappa^+\).
			
			Now suppose \(\Thorn_\beta^{\mathrm{V}}>\kappa\).  Let \(f:\Thorn_\beta^\omega\rightarrow \Thorn_{\beta+1}^{\mathrm{V}}\) be in \(\mathrm{V}[G]\).  Fixing a name \(\dot f\) for \(f\), we may consider \(f':\Thorn_\beta^\omega\times \mathbb{P}\rightarrow \Thorn_{\beta+1}^{\mathrm{V}}\) defined by
			\[f'(x,p)=\begin{cases}
				\alpha&\text{if }p\Vdash\dblqm{\dot f(\check x)=\check\alpha}\text{,}\\
				0&\text{otherwise.}
			\end{cases}\tag{\(*\)}\]
			It follows that if \(f\) is surjective then \(f'\) is surjective in \(\mathrm{V}\) with domain equivalent to \(\Thorn_\beta^\omega\times 2^\kappa\), which we may regard as a subset of \(\Thorn_\beta^{\kappa}\subseteq \Thorn_\beta^{<\Thorn_\beta}\).  There can be no such surjection in \(\mathrm{V}\) by the strong regularity of \(\Thorn_{\beta+1}^{\mathrm{V}}\).  Hence \(\mathrm{V}[G]\) has no such surjection \(f:\Thorn_\beta^\omega\rightarrow \Thorn_{\beta+1}^{\mathrm{V}}\) and so \(\Thorn_{\beta+1}^{\mathrm{V}[G]}=\Thorn_{\beta+1}^{\mathrm{V}}\).
			
			\item If \(\lambda <\mu\) with \(\mu >\kappa^+\), then without loss of generality, \(\lambda \ge\kappa^+\).  By the same proof as with \autoref{changthorn1} (4), if \(\lambda <\mu\) has an unbounded \(f:\lambda^{<\lambda}\rightarrow \mu\) in \(\mathrm{V}[G]\), then we get an unbounded function \(f':\lambda^{<\lambda}\times 2^{\kappa}\rightarrow \mu\) defined using a name \(\dot f\) for \(f\) as with (\(*\)).  Since \(2^\kappa =\kappa^+\) in \(\mathrm{V}[G]\) by (1), and \(\lambda^{<\lambda}\ge \lambda  = \max(\lambda ,\kappa^+)\), by coding, this gives an unbounded \(g:\lambda^{<\lambda}\rightarrow \mu\), a contradiction.
				
			\item (4) tells us we preserve strong regularity for \(\mu >\kappa^+\).  We clearly preserve it for \(\mu \le\kappa\) by distributivity and the fact that we well-order \(\mu^{<\mu}\). So assume \(\mu =\kappa^+\) (which is calculated the same in \(\mathrm{V}\) and \(\mathrm{V}[G]\) by distributivity) is on the thorn sequence and is thus of the form \(\Thorn_{\beta+1}\) for some \(\beta\) and hence \(\Thorn_\beta\le\kappa\). By strong regularity, it follows that \(\Thorn_\beta^{<\Thorn_\beta}<\mathop{cof}(\Thorn_{\beta+1})\) in \(\mathrm{V}\).  In \(\mathrm{V}[G]\), we end up adding a well-order of \(\kappa^\kappa\) and hence of \(\Thorn_\beta^{\Thorn_\beta}\) of length \(\le\kappa^+\).  By distributivity, \(\kappa\) is still regular in \(\mathrm{V}[G]\) so that by regularity, \(\Thorn_\beta^{<\Thorn_\beta}\subseteq \kappa^{<\kappa}=\kappa <\mu\), as desired.\qedhere
		\end{enumerate}
	\end{proof}

	This allows us to continue to force with \(\mathrm{Add}(\Thorn_n^+,1)\) to calculate \(\Thorn_{n+1}\) and in fact \(\Thorn_{n+2}\) as \(\Thorn_{n}^{++}\).  And we can continue using this lemma with the only issue being hypothesis (b).  And this is where \(\Thorn_n\) being justified comes in, similar as with \autoref{changthorn1justified}.
		
	Let us deal with the case of \(\omega_2\) first.  In doing so, it will be useful to have the following which will be repeatedly used to extend \(\mathsf{DC}_{\kappa^+}\) for relations on \((\kappa^{++})^\omega\) to \(\mathsf{DC}_{\kappa^+}\) to relations on \(\mathcal{P}(\kappa^+)\).
	\begin{lemma}\label{surjchain}
		Let \(f:Y\rightarrow X\) be a surjection.  For \(R\subseteq X\times X\) let \(R'=\{\langle y,z\rangle :f(y)\mathrel{R}f(z)\}\).  Suppose \(\mathsf{DC}_{\tau}\) for relations on \(Y\) holds for \(\tau<\kappa\).  Then \(<\kappa\)-length \(R\)-chains have upper bounds iff \(<\kappa\)-length \(R'\)-chains have upper bounds.
	\end{lemma}
	\begin{lemma}\label{surjchain}
		Let \(R\subseteq X\times X\) be a relation, and define \(R'\) by \(y\mathrel{R'}z\) iff \(f(y) \mathrel{R} f(z)\), where \(f:Y\rightarrow X\) is a surjection.  Suppose that for \(\tau<\kappa\), \(\mathsf{DC}_{\tau}\) for relations on \(Y\) holds.  Then \(R\) is \(<\kappa\)-closed iff \(R'\) is.
	\end{lemma}
	\begin{proof}
		For one direction, suppose \(<\kappa\)-length \(R\)-chains can be extended.  Let \(\langle y_\alpha:\alpha<\tau\rangle\) for \(\tau<\kappa\) be an arbitrary \(R'\)-chain.  By definition of \(R'\), this means \(\langle f(y_\alpha):\alpha<\tau\rangle\) is a \(<\kappa\)-length \(R\)-chain.  So there is some upper-bound \(x\in X\).  By surjectivity, \(x=f(y)\) for some \(y\) so that \(f(y_\alpha)\mathrel{R} f(y)\) for all \(\alpha<\tau\).  So \(y\) is an upper-bound for \(\langle y_\alpha:\alpha<\tau\rangle\).
		
		For the other direction, we use \(\mathsf{DC}_{\tau}\) on \(Y\) for \(\tau<\kappa\).  For any \(<\kappa\)-length \(R\)-chain \(\langle x_\alpha:\alpha<\tau\rangle\), by \(\mathsf{DC}_{\tau}\), we can find a \(\tau\)-length \(R'\)-chain \(\langle y_\alpha:\alpha<\tau\rangle\) where \(f(y_\alpha)=x_\alpha\).  If this has an upper-bound \(y\) then \(f(y)\) is an upper-bound for \(\langle x_\alpha:\alpha<\tau\rangle\).\qedhere
	\end{proof}
	
	An easy corollary is the following.
	
	\begin{corollary}\label{surjDC}
		Assume \(\mathsf{DC}_{\kappa}\) for relations on \(Y\).  Suppose \(Y\twoheadrightarrow X\).  Then \(\mathsf{DC}_{\kappa}\) holds for relations on \(X\).
	\end{corollary}
	
	Now we may proceed generally as follows and assume we are forcing with \(\mathrm{Add}(\kappa^+,1)\).  Here the jump from \(\Thorn_0=\omega\) to \(\Thorn_1=\omega_2\) poses a slight problem where we cannot assume \(\kappa\) is on the thorn sequence.  Nevertheless, this doesn't pose a big problem, as \(\kappa=\omega_1\) has \(\kappa^{<\kappa}\approx\omega^\omega\), and we can reason with \(\kappa\) using the combinatorial results around \(\omega^\omega\).  For \(\kappa>\omega_1\), we without loss of generality should think of \(\kappa\) as on the thorn sequence.
	\begin{lemma}\label{changDClatersteps}
		\begin{enumerate}[label=\alph*.]
			\item Let \(\kappa>\aleph_0\) be a cardinal in \(\mathrm{L}(\mathrm{Ord}^\omega)\vDash\mathsf{ZF}+\mathsf{DC}\).
			\item Let \(\kappa^\omega\twoheadrightarrow\mathbb{P}_0\) in \(\mathrm{L}(\mathrm{Ord}^\omega)\).
			\item Let \(G_0*G_1\) be \(\mathbb{P}_0*\mathbb{P}_1=\mathbb{P}_0*\mathrm{Add}(\kappa^+,1)\)-generic over \(\mathrm{L}(\mathrm{Ord}^\omega)\).
			\item In \(\mathrm{V}=\mathrm{L}(\mathrm{Ord}^\omega)[G_0]\),
			\begin{enumerate}[label=\roman*.]
				\item Assume \(\mathsf{DC}_\kappa\) holds for relations on \(\mathcal{P}(\kappa)\).
				\item Suppose \(\kappa\) is regular and \(\kappa^{<\kappa}=\kappa\).
				\item Suppose \(\Thorn_{\alpha-1}\le \kappa<\Thorn_{\alpha}\) for some \(\alpha\ge 1\).
				\item Suppose \(\Thorn_{\alpha}=\Thorn_{\alpha}^{\mathrm{L}(\mathrm{Ord}^\omega)}\) and \(\Thorn_{\alpha+1}=\Thorn_{\alpha+1}^{\mathrm{L}(\mathrm{Ord}^\omega)}\) are strongly regular and justified \(\alpha\ge 1\).
				\item Suppose \(|\Thorn^{<\Thorn}|=\Thorn\) for \(\Thorn<\kappa\) on the thorn sequence.
			\end{enumerate}
		\end{enumerate}
		Then, in \(\mathrm{V}[G_1]=\mathrm{L}(\mathrm{Ord}^\omega)[G_0*G_1]\), \(\mathsf{DC}_{\kappa^+}\) holds for relations on \(\Thorn_{\alpha+1}^\omega\) and in fact on \(\mathcal{P}(\kappa^+)\).
	\end{lemma}
	\begin{proof}
		Because \(\Thorn_{\alpha-1}\le\kappa<\Thorn_{\alpha}\), we have \(\Thorn_{\alpha-1}^\omega\twoheadrightarrow\kappa\) and so \(\Thorn_{\alpha-1}^\omega\approx(\Thorn_{\alpha-1}^\omega)^\omega\twoheadrightarrow\kappa^\omega\).
		In particular, \(\aleph^*(\Thorn_{\alpha-1}^\omega)=\aleph^*(\kappa^\omega)=\Thorn_{\alpha}\) in \(\mathrm{V}\).
		
		Now consider \autoref{changthornlatersteps} with the \(\alpha\) there being \(\alpha+2\) from here.  The hypotheses (e) and (f) there hold by (ii), (iv), and (v) here.  The other hypotheses are immediate.  So the conclusion of \autoref{changthornlatersteps} tells us
		\begin{itemize}
			\item By (1), \(2^\kappa=\kappa^+\) so that \(\aleph^*(\kappa^\omega)=\aleph^*(\kappa)=\kappa^+\) in \(\mathrm{V}[G_1]\).
			\item By (3), \(\Thorn_\alpha^{\mathrm{V}}=\Thorn_\alpha^{\mathrm{V}[G_1]}=\kappa^+\) and \(\Thorn_{\alpha+1}^{\mathrm{V}}=\Thorn_{\alpha+1}^{\mathrm{V}[G_1]}\).
			\item By (5), \(\Thorn_{\alpha}\) and \(\Thorn_{\alpha+1}\) are still strongly regular.
		\end{itemize}
		This also tells us \((\kappa^{+})^{<\kappa^+}=\kappa^+\) so that \(\Thorn_{\alpha+1}=\kappa^{++}\) in \(\mathrm{V}[G_1]\).
		
		Let \(R\) be a relation on \((\kappa^{++})^\omega\) such that \(\le\kappa\)-length chains can be extended.  For each \(\xi <\kappa^{++}\), let \(\gamma_\xi\) be the least such that every chain in \(\xi^\omega\) has an upper bound in \(\gamma_\xi^\omega\).
		\begin{claim}\label{changDClaterstepscl1}
			\(\gamma_\xi <\kappa^{++}\) whenever \(\xi <\kappa^{++}\).
		\end{claim}
		\begin{proof}
			Let \(\xi <\kappa^{++}\) be arbitrary.  There is surjection \(f:\Thorn_\alpha^\omega\rightarrow \xi^\omega\) in \(\mathrm{V}[G_1]\). Take \(g:\xi^\omega\rightarrow \kappa^{++}\) to be defined such that \(x\) has an upper bound in \(g(x)^\omega\) for \(g(x)<\kappa^{++}\) least.  It follows that \(g\) must be bounded in \(\kappa^{++}\) since otherwise \(g\circ f:\Thorn_\alpha^\omega\rightarrow \Thorn_{\alpha+1}\) would violate the strong regularity of \(\Thorn_{\alpha+1}\) in \(\mathrm{V}[G_1]\).\qedhere
		\end{proof}
		As a result, we can find a \(\hat\beta<\kappa^{++}\) of cofinality \(\kappa^+\) such that if \(\xi <\hat\beta\) then \(\gamma_\xi <\hat\beta\) just as with (6) of \autoref{changthorn1}: take the supremum of the sequence defined by recursion \(\beta_{0}=\gamma_0\), \(\beta_{\xi +1}=\gamma_{\beta_\xi}\), and with supremums at limit stages \(\le\kappa^+\).  In other words, define \(\hat\beta=\sup_{\xi <\kappa^{+}}\beta_\xi\).
		
		In \(\mathrm{V}[G_1]\), \(\kappa^+\twoheadrightarrow(\kappa^+)^\omega\twoheadrightarrow\hat\beta^\omega\).  So \(\hat\beta^\omega\) can be well-ordered which implies we can construct a \(\kappa^+\)-length branch of \(R\): at each successor stage, we just pick the least \(\vec{\alpha}\in \hat\beta^\omega\) that extends our current branch, and at limit stages, we just take unions.  Thus \(\mathsf{DC}_{\kappa^+}\) holds for relations on \((\kappa^{++})^\omega\) in \(\mathrm{V}[G_1]\).
		
		\(\Thorn_{\alpha+1}\) and \(\Thorn_{\alpha}\) being justified tells us that we have \(\mathsf{DC}_{\kappa^+}\) on \(\mathcal{P}(\kappa^+)\) in \(\mathrm{V}[G_1]\), which we now show.
		\begin{claim}\label{changDClaterstepscl2}
			In \(\mathrm{V}[G_1]\), \((\kappa^{++})^\omega\twoheadrightarrow \mathcal{P}(\kappa^+)^{\mathrm{V}[G_1]}\).
		\end{claim}
		\begin{proof}
			In forcing with any \(\mathbb{P}\), subsets of \(\tau\in \mathrm{Ord}\) are given by \(\mathbb{P}\)-names for subsets of \(\tau\), which can be regarded as elements of \(\mathcal{P}(\tau\times \mathbb{P})\).  So it suffices to show in \(\mathrm{L}(\mathrm{Ord}^\omega)\) that
			\[\Thorn_{\alpha+1}^\omega\twoheadrightarrow \mathcal{P}(\Thorn_{\alpha}\times \mathbb{P}_0\times \mathbb{P}_1)\text{.}\]
			We can regard \(\mathbb{P}_1=\mathcal{P}(\kappa)^{\mathrm{V}}\).  Elements of \(\mathcal{P}(\kappa)^{\mathrm{V}}\) can be regarded as elements of \(\mathcal{P}(\kappa\times \mathbb{P}_0)^{\mathrm{L}(\mathrm{Ord}^\omega)}=\mathcal{P}(\kappa\times \mathbb{P}_0)^{\mathrm{L}(\mathrm{Ord}^\omega)}\).  Now as \(\Thorn_\alpha\) is justified, \(\mathcal{P}(\Thorn_{\alpha-1}^\omega)\subseteq \mathrm{L}_{\Thorn_{\alpha}}(\Thorn_{\alpha}^\omega,X)\) for some \(X\subseteq \Thorn_{\alpha}\).  So in \(\mathrm{L}(\mathrm{Ord}^\omega)\),
			\[\Thorn_{\alpha}^\omega\approx \Thorn_{\alpha}^\omega\times \Thorn_{\alpha}\times X\twoheadrightarrow \mathcal{P}(\Thorn_{\alpha-1}^\omega)\twoheadrightarrow \mathcal{P}(\kappa^\omega)\approx\mathcal{P}(\kappa\times\kappa^\omega)\twoheadrightarrow \mathcal{P}(\kappa\times\mathbb{P}_0)\approx\mathbb{P}_1\text{,}\]
			Where this last surjection is due to (b).  Also, clearly \(\Thorn_\alpha^\omega\twoheadrightarrow\kappa^\omega\twoheadrightarrow\mathbb{P}_0\).
			
			Now since \(\Thorn_{\alpha+1}\) is justified, \(\mathcal{P}(\Thorn_{\alpha}^\omega)\subseteq \mathrm{L}_{\Thorn_{\alpha+1}}(\Thorn_{\alpha+1}^\omega,X)\) for some \(X\subseteq \Thorn_{\alpha+1}\).  So in \(\mathrm{L}(\mathrm{Ord}^\omega)\),
			\begin{align*}
				\Thorn_{\alpha+1}^\omega&\twoheadrightarrow\mathcal{P}(\Thorn_{\alpha}^\omega)\approx\mathcal{P}(\Thorn_{\alpha}^\omega\times \Thorn_{\alpha}^\omega\times \Thorn_{\alpha}^\omega)\\
				&\twoheadrightarrow\mathcal{P}(\Thorn_{\alpha}\times\mathbb{P}_0\times\mathbb{P}_1)\text{.}
			\end{align*}
			Thus in \(\mathrm{V}[G_1]\), since \(\Thorn_{\alpha}\) and \(\Thorn_{\alpha+1}\) are the same in \(\mathrm{L}(\mathrm{Ord}^\omega)\), \(\mathrm{V}\), and \(\mathrm{V}[G_1]\), we get \(\Thorn_{\alpha+1}^\omega\twoheadrightarrow \mathcal{P}(\Thorn_{\alpha})\).\qedhere
		\end{proof}
		
		So \autoref{surjDC} implies \(\mathsf{DC}_{\kappa^+}\) on \(\mathcal{P}(\kappa^+)\) in \(\mathrm{V}[G_1]\).\qedhere
	\end{proof}
	
	This allows us to continually force more and more \(\mathsf{DC}\) over \(\mathrm{L}(\mathrm{Ord}^\omega)\), but only up to a point.  We can get \(\mathsf{DC}_{\aleph_n}\) over \(\mathcal{P}(\aleph_n)\) for arbitrarily large \(n\), but continuing to the infinite case has problems.  One such problem is what happens when we consider an infinite iteration \(\bigast_{n<\omega}\mathrm{Add}(\aleph_n^+,1)\).  If the support is finite, we might end up adding real numbers and change \(\Thorn_1\).  But regardless of support, it's not clear that the tail posets will preserve the properties we want them to, as they will be too large to use dependent choice: at any stage we only have \(\mathsf{DC}_\kappa\) over \(\mathcal{P}(\kappa)\), but the tail poset would have size \(\ge\aleph_\omega>\kappa\) for any \(\kappa<\aleph_\omega\).
	
	Nevertheless, we can continue this iteration for arbitrarily many finite steps, giving \autoref{mainthmfinitethorn}.
	\begin{theorem}\label{thmfinitethorn}
		Work in \(\mathrm{V}=\mathrm{L}(\mathrm{Ord}^\omega)\vDash\mathsf{ZF}+\mathsf{DC}\), and let \(n<\omega\). Suppose \(\Thorn_i\) is justified and strongly regular for each \(i\in \omega\). Let \(\mathbb{P}_n=\bigast_{i<n}\dot{\mathbb{Q}}_i\) be the iteration where \(\mathbb{Q}_i=\mathrm{Add}(\aleph_i^+,1)\), and let \(G_n\) be \(\mathbb{P}_n\)-generic over \(\mathrm{V}\). Then in \(\mathrm{V}[G_n]\),
		\begin{enumerate}[label=\roman*.]
			\item \(\mathsf{DC}_{\kappa}\) holds for relations on \(\mathcal{P}(\kappa)\) for all \(\kappa\le\aleph_n\).
			\item \(\Thorn_i^{\mathrm{V}}=\Thorn_i^{\mathrm{V}[G_n]}\) for all \(i\in \omega\).
			\item \(\Thorn_0=\omega\), and \(\Thorn_{i}=\aleph_i^+\) for all \(0<i\le n\).
			\item \(\kappa^+=2^\kappa\) for \(\kappa<\aleph_n\).
			\item Each \(\Thorn_i\) is still strongly regular for \(n\in \omega\).
		\end{enumerate}
	\end{theorem}
	\begin{proof}
		Proceed by induction on \(n\) to show (i)--(v) hold in \(\mathrm{V}[G_n]\) in addition to
		\begin{enumerate}
			\item[vi.] \(\aleph_i^{<\aleph_i}=\aleph_i\) for each \(i\le n\).
		\end{enumerate}
		Note that with \(\mathbb{P}_n\), for \(n>0\), the last poset forced with is \(\mathrm{Add}(\aleph_n,1)\).  For \(n=0\), \(\mathbb{P}_0\) is trivial, and all of the above hold obviously.  For \(n=1\), these all hold by \autoref{changthorn1} with (vi) holding due to \(\mathsf{CH}\).  Assume (i)--(vi) hold inductively for \(n\ge 1\).  For \(n+1\), \(\mathrm{V}[G_{n+1}]=\mathrm{V}[G_n*g_{n+1}]\) where \(g_{n+1}\) is \(\mathrm{Add}(\aleph_{n+1},1)\)-generic over \(\mathrm{V}[G_n]\).  We first use \autoref{changthornlatersteps} with ground model \(\mathrm{V}[G_n]\) and forcing with \(\mathrm{Add}(\kappa^+,1)\) with \(\kappa=\aleph_n\), which means confirming the hypotheses of the lemma.
		
		The hypotheses (a)--(d) of \autoref{changthornlatersteps} hold easily.  The hypothesis (e) holds with \(\alpha=\omega\): in \(\mathrm{V}[G_n]\), for any \(i<\omega\) such that \(\aleph_n<\Thorn_i\), \(\Thorn_i\) is strongly regular.  We inductively get (f), which states that in \(\mathrm{V}[G_n]\), for any \(i\), \(\Thorn_i^\omega\) can be well-ordered to have size \(\le \aleph_n\) whenever \(\Thorn_i<\aleph_n\).  To see this, by (iii) in \(\mathrm{V}[G_n]\), if \(\Thorn_i<\aleph_n\) then either \(\Thorn_i=\aleph_0\) so that \(\Thorn_i^\omega\) has size \(\aleph_1\le\aleph_{n}\), or else \(\Thorn_i=\aleph_{i+1}<\aleph_n\) has by (vi) of \(\mathrm{V}[G_n]\) that \(\aleph_j^{<\aleph_j}=\aleph_j<\aleph_n\) for \(j=i+1<n\).
		
		The conclusion of \autoref{changthornlatersteps} for \(\kappa=\aleph_{n}\) and \(\alpha=\omega\) is that in \(\mathrm{V}[G_{n+1}]\),
		\begin{itemize}
			\item By (1), we get (iv): \(2^{\aleph_{n}}=\aleph_{n+1}\).
			\item This also implies (vi) since inductively \(\aleph_i^{<\aleph_i}=\aleph_i\) for \(i\le n\) (distributivity of \(\mathrm{Add}(\aleph_{n+1},1)\) doesn't change this) and \(\aleph_{n+1}^{<\aleph_{n+1}}=(2^{\aleph_n})^{\aleph_n}=\aleph_{n+1}\).
			\item By (3), we get (ii): \(\Thorn_i^{\mathrm{V}[G_n]}=\Thorn_i^{\mathrm{V}[G_{n+1}]}\) which is inductively \(\Thorn_i^{\mathrm{V}}\) for \(i<\omega\).
			\item The above two points give (iii): \(\Thorn_{n+1}=|\aleph_{n+1}^{\omega}|^+=|\aleph_{n+1}^{<\aleph_{n+1}}|^+=\aleph_{n+1}^+\), and we still have \(\Thorn_i=\aleph_{i+1}\) for \(i\le n\) by distributivity.
			\item By (5) we get (v): in \(\mathrm{V}[G_n]\), inductively \(\aleph_n=\Thorn_{n-1}\) by (iii), and \(\Thorn_{n-1}\) is strongly regular by (v) and by (vi).  We know \(\aleph_{n}^{<\aleph_n}=\aleph_n\), so the hypothesis of (5) holds and so \(\mathrm{Add}(\aleph_{n+1},1)\) preserves strong regularity.
		\end{itemize}
		So we have (ii)--(vi) for \(n+1\), and all that remains in the induction is showing (i): that \(\mathsf{DC}_{\aleph_{n+1}}\) holds for relations on \(\mathcal{P}(\aleph_{n+1})\).  For \(n+1=2\), consider \autoref{changDClatersteps} with \(\alpha=n=1\), \(\kappa=\omega_1\), and \(\mathbb{P}_0\) of the lemma as \(\mathrm{Add}(\omega_1,1)=\mathbb{P}_1\) here.  The hypothesis (b) of \autoref{changDClatersteps} is clear then: \(\kappa^\omega=\omega_1^\omega\approx\mathbb{P}_1\) in \(\mathrm{L}(\mathrm{Ord}^\omega)\).  The other hypotheses are clear by the inductive hypothesis on \(\mathrm{V}[G_1]\) so \(\mathsf{DC}_{\omega_2}\) holds in \(\mathrm{V}[G_2]\) for relations on \(\mathcal{P}(\omega_2)\).
		
		For \(n+1>2\), consider \autoref{changDClatersteps} with
		\begin{itemize}
			\item \(\alpha=n\ge 1\) and \(\kappa=\Thorn_{n}>\aleph_0\), and
			\item the poset \(\mathbb{P}_0\) of the lemma as \(\mathbb{P}_n\).
		\end{itemize}
		The hypotheses of \autoref{changDClatersteps} hold by the inductive hypothesis on \(\mathrm{V}[G_n]\) with the only work being to show (b), that \(\Thorn_n^\omega\twoheadrightarrow \mathbb{P}_n\).  This follows because we can embed \(\mathbb{P}_n\) into \(\mathcal{P}(\Thorn_{n-1})\) of \(\mathrm{L}(\mathrm{Ord}^\omega)\).  Since \(\Thorn_n\) is justified, we have \(\Thorn_n^\omega\twoheadrightarrow \mathcal{P}(\Thorn_{n-1})\twoheadrightarrow \mathbb{P}_n\).  So the conclusion that \(\mathsf{DC}_{\Thorn_n^+}\) holds for relations on \(\mathcal{P}(\Thorn_n^+)=\mathcal{P}(\aleph_{n+2})\) gives (i).\qedhere
	\end{proof}

	As stated before, there are problems with continuing this iteration to infinitely many steps, but even if we could alleviate these problems and go to \(\mathrm{V}[G]\) where \(\Thorn_i=\aleph_{i+1}\) for all \(i<\omega\) with \(\mathsf{DC}_\kappa\) for relations on \(\mathcal{P}(\kappa)\) for all \(\kappa<\aleph_\omega\), there are still other obstacles in continuing.  In \(\mathrm{V}[G]\), we also of course have \(\Thorn_\omega=\aleph_\omega\), but it's not clear how we can go beyond this because we don't have \(\mathsf{DC}_{\omega_\omega}\) over relations on \(\mathcal{P}(\omega_\omega)\), and the calculation of \(\aleph_\omega^\omega\) is also unknown---one would expect due to König's lemma that if this is well-ordered, \(\aleph_\omega^{\aleph_0}>\aleph_{\omega}\).\footnote{One might expect that \(\Thorn_{\omega+1}\) then becomes \(\Thorn_{\omega}^{++}\) similar to \(\Thorn_{1}=\Thorn_0^{++}\) as above, but this is just conjecture.} With the lack of \(\mathsf{DC}_{\omega_\omega}\), the distributivity of \(\mathrm{Add}(\omega_{\omega+1},1)\) is called into question, and overall, the arguments start to fall apart.
	
	So it remains an open question how to deal with \(\aleph_\omega\) more generally in the Chang model, a problem also encountered 
	in \cite{LarsonSargsyan}.  Another open question is just how strong is the statement that \(\Thorn_n\) is strongly regular and justified for \(n<\omega\)?

\end{document}